\title{Quantitative universality for products of i.i.d. random matrices}
\author{Nikita Lvov\footnote{nikita.lvov@mail.mcgill.ca}}

\documentclass{article}

\usepackage{ifxetex,ifluatex}
\if\ifxetex T\else\ifluatex T\else F\fi\fi T%
  \usepackage{fontspec}
\else
  \usepackage[T1]{fontenc}

  \usepackage{blkarray, bigstrut}
  \usepackage[utf8]{inputenc}
  
  \usepackage{amsthm}
  \usepackage{thmtools}
  \usepackage{mathtools}

  \usepackage{lmodern}
  \usepackage{amssymb}
  \usepackage{amsfonts}
  \usepackage{tikz-cd}
  \usepackage{amsthm}
  \usepackage{xfrac}
  \usepackage{mathtools}
  \usepackage{multirow}
  \usepackage{mathrsfs}
  \usepackage{comment}

  \newtheorem*{corollary}{Corollary}

  \newcommand{\probP}{\text{I\kern-0.15em P}}
  \newcommand{\probE}{\text{I\kern-0.15em E}}

  \numberwithin{equation}{section}

  \theoremstyle{remark}
  \newtheorem{remark}{Remark}[subsection]
  \theoremstyle{remark}

  \newtheorem*{observation}{Observation}

  \newcommand{\Z}{\mathbb{Z}}

  \excludecomment{mysection}
  \excludecomment{mymysection}
  \excludecomment{maybeinclude}

\newcommand{\probM}{\mathcal{M}}
\newcommand{\probU}{\mathcal{U}}

\usepackage{mdframed}

\theoremstyle{theorem}

  \newtheorem{proposition}{Proposition}

\newtheorem{theorem}{Theorem}
\newtheorem*{extratheorem}{Theorem}

\newenvironment{ftheo*}
  {\begin{mdframed}\begin{theorem*}}
  {\end{theorem*}\end{mdframed}}

\newtheorem*{goal}{Goal}

  \usepackage{enumitem}

   \DeclareSymbolFont{bbold}{U}{bbold}{m}{n}
   \DeclareSymbolFontAlphabet{\mathbbold}{bbold}

   \theoremstyle{remark}

\fi

 \theoremstyle{definition}

 \newcommand{\comm}[1]{} 

 \usepackage{cite}
 \newcommand{\commm}[1]{}
 
 \newcommand{\chapter}{section}

  \usepackage{
  hyperref,
  cleveref,
  }
\begin{document} 
\maketitle
\abstract{Using estimates established in a previous paper, we prove quantitative universality results for cokernels of products of i.i.d. random matrices and for flags associated to $k$-tuples of random i.i.d. matrices, over finite local rings. In the case when the ring is a quotient of $\Z_p$, this gives a quantitative analogue of some previous results of Huang, Nguyen and Van Peski.}
\newcommand{\probG}{\mathcal{G}}
\newcommand{\A}{\probM}
\newcommand{\U}{\mathcal{U}}

    \newcommand{\bigosum}{\bigoplus}
   \newcommand{\osum}{\oplus}
   \newcommand{\sgn}{sgn}
   \newcommand{\Sset}{S}
   \newcommand{\PP}{\mathcal{P}}
   \newcommand{\ProbG}{\mathcal{G}}
   \newcommand{\ProbM}{\mathcal{M}}
   \newcommand{\ProbN}{\mathcal{N}}
      \newcommand{\probN}{\mathcal{N}}
   \newcommand{\ProbU}{\mathcal{U}}

   \newcommand{\ltwo}[1]{l^2\left( #1\right)}
      \newcommand{\linf}[1]{l^{\infty}\left( #1\right)}
            \newcommand{\lqnorm}[1]{l^{q}\left( #1\right)}

\section{Introduction}

In \cite{arxivtwo}, we have derived quantitative universality statements for the
distribution of cokernels, determinants and spans of i.i.d. random
matrices over a general local finite ring $R$. In this note, we show how the results of \cite{arxivtwo} also imply quantitative universality statements for the distributions of:

\begin{itemize}
  \item[(A)]   The cokernel of a product of i.i.d. random matrices,
  \item[(B)]   The determinant of a product of i.i.d. random matrices,
  \item[(C)]   The flag of modules associated to a $k$-tuple of i.i.d. random matrices. 
\end{itemize}

These arguments reprove, by a different method, some previous results of Roger
Van Peski, Hoi Nguyen and Yifeng Huang: \cite{universalityproducts}\cite{rescaledmomentmethod}\cite{flaguniversality} for $p$-adic random matrices. 
\subsection{Previous work on products of random matrices and universality.}
Van Peski has a very interesting body of work related
to products of $p$-adic random matrices: \cite{Van_Peski_2021}\cite{hlpolynomialsboundariesprm}\cite{poissonsea}\cite{locallimits}\cite{padicbrownianmotion}. Nguyen and Van Peski proved a
universality result for cokernels of products of a fixed number of matrices in \cite{universalityproducts}, using the moment method. In \cite{rescaledmomentmethod} the same authors also proved a universality result for cokernels of products of a growing number of matrices, using a renormalized moment method. Huang determined the distribution of flags associated to products of square and rectangular matrices in \cite{yifengflags}. Most recently, Huang, Nguyen and Van Peski proved universality for flags of cokernels associated to $k$-tuples of square random matrices in \cite{flaguniversality}.  

\subsection{Outline} We give a brief outline of this note.
First of all, we can give an equivalent reformulation of \cite[Theorem 1.1 (C)]{arxivtwo}, as the inequality (\ref{eqn: bound from arxivtwo}). This inequality can be expressed symbolically in the following form: $
\mathcal{M} \mathcal{G} \sim \mathcal{U}
$, where $\sim$ means that the total variation distance between the two probability distributions on matrices is bounded by $O(e^{-cN})$, for some fixed value of $c$. This implies that: $
\mathcal{M}_1 \mathcal{M}_2 \mathcal{G} \sim
\mathcal{M}_1 \mathcal{U}_2 =
\mathcal{M}_1 \mathcal{G} \mathcal {U}_2 \sim
\mathcal{U}_1 \mathcal{U}_2
$, and more generally,  $\mathcal{M}_1 \mathcal{M}_2 \cdots \probM_k \mathcal{G} 
 \sim
\mathcal{U}_1 \mathcal{U}_2 \cdots \probU_k$. This allows us to deduce the following theorem:

\begin{extratheorem}[\autoref{prop: mainargument} in \S \ref{sec: mainargument}] For each $i$, suppose that $\mathcal{M}_i$ is an i.i.d. random matrix, with values in a finite local ring $R$. Assume that the support of the distribution of the entries of $\A_i$ 
\begin{itemize}
\item[(a)] contains two elements that differ by a unit,
\item[(b)] is not contained in the affine translate of a proper subring of $R$.
\end{itemize}

Then: 
\[
\A_1 \A_2 ... \A_k \probG   \sim   \mathcal{U}_1 \mathcal{U}_2 ... \mathcal{U}_k
\]

\end{extratheorem}

\begin{corollary} $coker( \A_1 \A_2 ... \A_k )   \sim  coker( \U_1 \U_2 ... \U_k )$

\end{corollary}

\begin{remark} $GL_N(R)$ can be replaced with $SL_N(R)$, and more generally by $E_N$, the subgroup generated by matrices that have $1$'s on the diagonal and at most one other non-zero entry\footnote{The notation $E_N$ is from mathoverflow.net/questions/59884/ and the references therein.}.
\end{remark}

We will slightly refine the above argument in \S \ref{sec: secondmainargument} to also deduce universality
for flags of cokernels
\[
coker(M_1) \leftarrow coker(M_1 M_2) \leftarrow coker(M_1 M_2 M_3) \leftarrow \cdots
\]
In particular we get universality for the sequence of cokernels:
\[
coker(M_1)
\, , \,
coker(M_1 M_2)
\, , \,
coker(M_1 M_2 M_3)
\, , \,
\cdots
\]
The most general statement in this vein is \autoref{prop: secondmainargument}. Finally, we will explain that the above results can be naturally generalized to
$k$-tuples of rectangular matrices, and to matrices over finite local rings, other than $\Z / p^a \Z$.

\subsection{Acknowledgements}

The author would like to thank Elia Gorokhovsky, Lily Levitsky, Hoi Nguyen, Roger Van Peski and Melanie Wood for helpful discussions related to the subject of this paper. The author would also like to particularly thank Hoi Nguyen for organizing a very fruitful meeting in May 2026, that inspired the ideas in this paper.
\newline
\newline
\noindent
AI was not used in the course of this project.

%
   
 \section{Universality results for products of i.i.d. random matrices}

\subsection{Restatement of \cite[Theorem 1.1 (C)]{arxivtwo}}
\label{sec: statement of main estimate}

Let $\probM_{m,n}$ be an $m \times n$ i.i.d. random matrix over a
finite local ring $R$. Let $\epsilon>0$ and let $\xi$ be the
distribution of the entries of $\probM_{m,n}$. Let $\mathfrak{m}$ be
the maximal ideal of $R$ and let $p$ be the characteristic of the
residue field, $R / \mathfrak{m}$. 
\newline
\newline
\noindent
\renewcommand{\probG}{\mathcal{G}}
Let $\probU_{m,n}$ an $m \times n$ random matrix over $R$ whose
entries are independent and uniformly random. Let $\probG_{n,n}$
denote a Haar random element in $E_n(R) \subset SL_n(R)$, the
subgroup of $SL_n(R)$ that is generated by transvections. Let
$d_{TV}$ denote the total variation distance between probability
measures.
\newline
\newline
Finally, suppose that $\xi$ is not concentrated on an affine translate of a proper subring of $R$.
\newline
\newline
\noindent
Then, we have the following bound on the total variation distance
between two matrix distributions, for any $\epsilon$:

\begin{equation}
\label{eqn: bound from arxivtwo}
d_{TV}
\Big(
\,(\probM_{m,n} \probG_{n,n})
\,,\,
\probU_{m,n}
\,\Big)<
\end{equation}
\[
K(R,\xi,\epsilon)
e^{\epsilon (m+n)} 
(\#R)^{\max(m-n,0)}
\left[
\ltwo{\xi \mod \mathfrak{m}}^m
+
\linf{\xi \mod \mathfrak{m}}^n
+
\frac{1}{p^m}
\right]
\]
\begin{remark}
By $\lqnorm{\xi \mod \mathfrak{m}}$, we mean the $l^q$ norm of the distribution of the random variable $(\xi \mod \mathfrak{m})$. It follows from the definition that $\lqnorm{\xi \mod \mathfrak{m}}\leq 1$ with equality if and only if the distribution of $\xi$ is concentrated on a single value modulo $ \mathfrak{m}$. Furthermore, $\linf{\xi \mod \mathfrak{m}}\leq \ltwo{\xi \mod \mathfrak{m}}$. 
\end{remark}
\begin{remark}
In the above definition of $\probG_{n,n}$, we can evidently replace $E_n(R)$ by any larger subgroup of
$GL_n(R)$. For  example, we can replace it by
$SL_n(R)$ or $GL_n(R)$.
\end{remark}
\subsubsection{Special cases}
Firstly, for the special case $R \cong \mathbb{Z} / p^a \mathbb{Z}$, we have:
\begin{equation}
\label{eqn: Main Theorem Part 2}
d_{TV}
\Big(
\,(\probM_{m,n} \probG_{n,n})
\,,\,
\probU_{m,n}
\,\Big)<
\end{equation}
\[
K(a, \xi, \epsilon) 
e^{\epsilon (m+n)} 
p^{\max(a(m-n),0)}
\left[
\ltwo{\xi \mod p}^m
+
\linf{\xi \mod p}^n
\right]
\]

Secondly, for the special case when $R \cong \mathbb{Z} / p^a
 \mathbb{Z}$ and when the random matrix is square, we have
\begin{equation}
\label{eqn: dtvsquare}
d_{TV}
\Big(
\,(\probM_{N,N} \probG_{N,N})
\,,\,
\probU_{N,N}
\,\Big)
<
C (a, \xi, \epsilon)
e^{\epsilon N}
\left[
\ltwo{\xi \mod p}
\right]^N
\end{equation}

 \subsubsection{Estimate of $K(\xi,R,\epsilon)$}
 \label{sec: kestimate}
The magnitude of $K(\xi,R,\epsilon)$ will not matter very much for the sequel. However, to give a rough bound on the size, we give a non-optimal equality for $K(\xi,R,\epsilon)$:

\[
K(\xi,R,\epsilon) \leq
\]
\begin{equation}
\label{eqn: kconstant bound}
\leq \left( \#\{\text{additive subgroups of $R$}\}\right)^2
(\#R)^T \max_{j>0}
\left(
\frac{j^T+1}
{ (1+\epsilon)^j},
\right)
\end{equation}
where $T$ is defined as the product of $log_p(\#R)$ and
\[
\left\lceil \frac{log(\epsilon/\#R)}{log(\text{largest absolute value of Fourier coefficients of $\xi$, not equal to $1$})} \right\rceil.
\]
We clarify the meaning of the denominator in the preceding expression. The random variable $\xi$ is an $R$-valued random variable, i.e. as a random variable valued in a finite abelian group. Therefore, it has $\#R$ Fourier coefficients, some of which may have absolute value $1$. We are interested in a Fourier coefficient with the largest absolute value, aside from those that have absolute value $1$.

\subsubsection{Estimate of $C(\xi,a,\epsilon)$}

From \S  \ref{sec: kestimate}, we can also deduce the following non-optimal bound on $C(\xi,a,\epsilon)$.
\begin{equation}
\label{eqn: cconstant bound}
C(\xi,a,\epsilon)\leq
2(a+1)^2
p^{a^2T} \max_{j>0}
\left(
\frac{j^{aT}+1}
{ (1+(\epsilon/2))^j},
\right)
\end{equation}
where $T\in \mathbb{N}$ is a natural number such that for all $m$ satisfying $p^a\nmid m$, the following inequality holds:
\[
\left| \probE\left[ exp\left(2 \pi  i\frac{ m}{p^a}\xi \right) \right]^{}\right|^{T}  \leq \frac{\epsilon}{2p^a}
\]


\newcommand{\dtv}[2]{d_{TV}\Big(#1, #2  \Big)}
\subsection{First consequence}
\label{sec: mainargument}
First of all, we will use the following proposition, that we give without proof. The notation used below is explained at the beginning of \S \ref{sec: statement of main estimate}.
\begin{proposition}
\label{prop: tv distance}
If $ \probN_1 , \probN_2 , \probN_3 $ are independent random matrices, then:
\[
\dtv {(\probN_3 \probN_1) }{(\probN_3 \probN_2) }
\leq
d_{TV}( \probN_1 , \probN_2 ) \geq \dtv {(\probN_1 \probN_3 ) }{( \probN_2 \probN_3) }
\] 
\end{proposition}

The main result of this subsection is the following theoirem.
\begin{theorem}
\label{prop: mainargument}
\begin{equation}
\dtv{
(
\probM_1 \probM_2 \cdots \probM_k \probG
)
}{
(
\probU_1 \probU_2 \cdots \probU_k
)
}<
\end{equation}
\[
<k
C(a,\xi, \epsilon)
e^{\epsilon N}
\ltwo{ \xi \mod  p}^N\]
\end{theorem}

\begin{proof}
Given two random matrices, we will write $\mathcal{N}_1 \sim \mathcal{N}_2$ if $d_{TV} (\mathcal{N}_1 , \mathcal{N}_2)$ is bounded above by the right hand side of (\ref{eqn: dtvsquare}). By the inequality (\ref{eqn: dtvsquare}), we have $\probM_k \probG \sim   \probU_k$, hence by \autoref{prop: tv distance}, 
\[
\probM_1 \probM_2 \cdots \probM_k \probG \sim \probM_1 \probM_2 \cdots \probM_{k-1 } \probU_k
\]
By the invariance of the Haar measure, the random matrix on the right has the same distribution as:
\[
 \probM_1 \probM_2 \cdots \probM_{k-1 } \probG \probU_k \sim  \probM_1 \probM_2 \cdots \probM_{k-2 } \probU_{k-1} \probU_k 
\]
where we have again used the inequality (\ref{eqn: dtvsquare}) and  \autoref{prop: tv distance}. By iterating this argument, we get \autoref{prop: mainargument}.

\end{proof}

\begin{corollary}
Let $f$ be any function on $N \times N$ matrices that is right
invariant by the action of $SL_N(
\mathbb{Z}/p^a \mathbb{Z}
)$. Then:
\[
\dtv
{
f(\probM_1 \probM_2 \cdots \probM_k)
}
{
f(\probU_1  \probU_2  \cdots  \probU_k)
}<k
C(a,\xi, \epsilon)
e^{\epsilon N}
\ltwo{ \xi \mod  p}^N
\]
\end{corollary}

\begin{remark}
In the above corollary, we can take $f$ to be the cokernel, the span
or the determinant. Hence, the corollary gives a quantitative
universality result for the cokernel, the span and the determinant of
products of i.i.d. random matrices.
\end{remark}



\subsection{Second consequence - a refinement of the previous estimate}
\label{sec: secondmainargument}
Again, the notation used below is explained at the beginning of \S \ref{sec: statement of main estimate}.
\begin{theorem}
\label{prop: secondmainargument}
Let
\begin{itemize}
\item$\{\probG_i\}$ be $k$ independent copies of a Haar random element
of $SL_N(\Z / p^a \Z)$, 
\item $\{ \probU_i \}$ be $k$ independent copies a Haar random element
of $Mat_{n,n} (\Z / p^a \Z)$,
\item $\{ \probM_i \}$ be $k$ independent i.i.d. random matrices with
entries in $\Z / p^a \Z$.
\end{itemize}
\[
d_{TV}\Big(
\left[        \,                           \probM_1   \probG_1\,,\,
 \probG_1^{-1}                \probM_2  \probG_2\,,\,
 \probG_2^{-1}               \probM_3  \probG_3\,,\,
 \cdots     \,  ,\,
 \probG_{k-1}^{-1}         \probM_k \probG_k\,
 \right]
,
\left[
\probU_1,
\probU_2,
\probU_3,
\cdots       ,
\probU_k
\right]
\Big)
\]
\begin{equation}
<k
C( \xi, a, \epsilon)
e^{\epsilon N}
\ltwo{ \xi \mod p}^N
\end{equation}
\end{theorem}

\begin{proof}
The proof of \autoref{prop: secondmainargument} is similar to the proof of \autoref{prop: mainargument}. By (\ref{eqn: dtvsquare}) and  a variation of \autoref{prop: tv distance}, we see that 
\[
\left[        \,                           \probM_1   \probG_1\,,\,
 \probG_1^{-1}                \probM_2  \probG_2\,,\,
 \probG_2^{-1}               \probM_3  \probG_3\,,\,
 \cdots     \,  ,\,
 \probG_{k-1}^{-1}         \probM_k \probG_k\,
 \right] \sim 
\]
\[
\left[        \,                           \probM_1   \probG_1\,,\,
 \probG_1^{-1}                \probM_2  \probG_2\,,\,
 \probG_2^{-1}               \probM_3  \probG_3\,,\,
 \cdots     \,  ,\,
  \probG_{k-2}^{-1}               \probM_{k-1}  \probG_{k-1}\,,\,
  \probG_{k-1}^{-1}    \probU_k\,
 \right]
\]
which, by the invariance of the Haar measure, has the same distribution as:
\[
\left[        \,                           \probM_1   \probG_1\,,\,
 \probG_1^{-1}                \probM_2  \probG_2\,,\,
 \probG_2^{-1}               \probM_3  \probG_3\,,\,
 \cdots     \,  ,\,
  \probG_{k-2}^{-1}               \probM_{k-1}  \probG_{k-1}\,,\,
    \probU_k\,
 \right]
\]
By iterating this argument, we get \autoref{prop: secondmainargument}
\end{proof}

\begin{corollary}
\label{cor: Nested Corollary}

Let $f$ be any function on $n \times n$ matrices that is invariant by
the    right action of $SL_n(\Z / p^a \Z )$. Then the total variation distance between the  distribution of

   \[
   \left[
   f(\probM_1),
   f(\probM_1 \probM_2),
   f(\probM_1 \probM_2 \probM_3),
   \cdots,
   f(\probM_1 \probM_2 \probM_3 \cdots \probM_k)
   \right]
\]
and
\[
   \left[
   f(\probU_1),
   f(\probU_1 \probU_2),
   f(\probU_1 \probU_2 \probU_3),
   \cdots,
   f(\probU_1 \probU_2 \probU_3 \cdots \probU_k)
   \right]
   \]
   is bounded above by 
   \begin{equation}
   \label{eqn: Nested Corollary}
   k
C( \xi, a, \epsilon)
e^{\epsilon N} 
\ltwo{ \xi \mod p}^N
   \end{equation}
   \end{corollary}

\begin{remark}
In the above corollary, we can take $f$ to be the span, the cokernel,
the determinant. The bound (\ref{eqn: Nested Corollary}) will hence give us a
quantitative universality result for the joint distribution of the
cokernels, determinants and spans.
\end{remark}

\begin{remark}
In particular, the joint distribution of the spans gives a distribution
on nested sequences of submodules of $(\Z / p^a \Z)^N$. By taking
quotients, we hence get a flag of quotients of $(\Z / p^a \Z)^N$.
Therefore, the bound (\ref{eqn: Nested Corollary}) can be used to deduce a
quantitative universality result for the distribution of flags.
\end{remark}

\begin{remark}
We note that in (\ref{eqn: Nested Corollary}, we can take $k$ increasing exponentially with $N$, and still get an exponentially decreasing bound on the total variation. For example, we can take $k=O(\ltwo{ \xi \mod p}^{-N/2})$. This implies that the total variation distance between the distribution of
   \[
   \left[
   coker(\probM_1),
   coker(\probM_1 \probM_2),
   \cdots,
   coker(\probM_1 \probM_2 \probM_3 \cdots \probM_k)
   \right]
\]
and the distribution of
\[
   \left[
   coker(\probU_1),
   coker(\probU_1 \probU_2),
   \cdots,
   coker(\probU_1 \probU_2 \probU_3 \cdots \probU_k)
   \right]
   \]
   is exponentially decreasing with $N$, even when $k$ is an exponentially increasing function of $N$.
\end{remark}




\subsection{Rectangular matrices}

All of the above results can be extended to $k$-tuples of rectangular
matrices, in a straightforward fashion.

From (\ref{eqn: Main Theorem Part 2}), we can deduce that the error
term, for $k$ matrices of dimension $\{m_i \times n_i \}_{i=1}^k$ will
have the form:

\[
C(a, \xi, \epsilon)
\sum_i
e^{\epsilon (m_i+n_i)}
p^{a(m_i-n_i)} 
\left[
\ltwo{\xi \mod  p}^{m_i} 
+
\linf{\xi \mod p}^{n_i}
\right]
\]



\subsection{Varying $a$.}
\label{sec:varyinga}
Typically, we have an i.i.d. random matrix over $\Z_p$.
We would like to reduce modulo $p^a$ to apply our universality
results.(Given a universality result modulo $p^a$ for every $a$, it is then
possible to deduce a universality result over $\Z_p$.) It might be useful to know
how the error term changes as we increase $a$. We make only the following observation.

\begin{observation}

The exponential term does not change as we increase $a$. However, one can see from \ref{eqn: cconstant bound} that the bound we can prove for the constant term grows rather quickly with $a$.

\end{observation}


\subsection{General local rings}

We can formulate analogous statements over a finite local ring in
place of $ \Z / p^a \Z $ and the results will continue to hold. We briefly summarize these results below:
\begin{itemize}
\item   Quantitative universality of (a) cokernels, (b) determinants,
and (c) spans of products of square i.i.d. random matrices.
\item   Quantitative universality for (a) cokernels, (b) determinants,
and (c) spans of partial products of $k$ independent square i.i.d.
random matrices.
\item   Analogous results when the matrices are not necessarily square.
  \end{itemize}

As in \S \ref{sec:varyinga}, we can suppose that we have a complete
local ring, and that we are considering quotients modulo $\mathfrak{m}^a$,
We can ask how the bounds change with $a$. We make the following observation.

\begin{observation}
The exponential term does not change as we increase $a$. However, one can see from \ref{eqn: kconstant bound} that the bound we can prove for the constant term grows rather quickly with $a$.
\end{observation}



%
%

\bibliographystyle{alpha}
\bibliography{ThesisBibliographyPrivetPrivet}

\end{document}